\pdfoutput=1
\documentclass[11pt]{article}
 
\usepackage[margin=1in]{geometry}
\usepackage{amsmath,amssymb,amsthm,mathtools}
\usepackage{booktabs}
\usepackage{array}
\usepackage{microtype}
\usepackage{xcolor}
\usepackage[colorlinks=true,linkcolor=blue!50!black,citecolor=blue!50!black,urlcolor=blue!50!black]{hyperref}
 
\usepackage[backend=biber,
            style=numeric-comp,
            sorting=nyt,
            maxbibnames=6,
            giveninits=true,
            doi=true, url=true, eprint=true]{biblatex}
\DeclareFieldFormat{titlecase}{#1}
 
\theoremstyle{plain}
\newtheorem{theorem}{Theorem}[section]
\newtheorem{proposition}[theorem]{Proposition}
\newtheorem{lemma}[theorem]{Lemma}
\newtheorem{corollary}[theorem]{Corollary}
\newtheorem{conjecture}[theorem]{Conjecture}
\theoremstyle{definition}
\newtheorem{definition}[theorem]{Definition}

\newtheorem{question}[theorem]{Question}
\theoremstyle{remark}
\newtheorem{remark}[theorem]{Remark}
 
\newcommand{\HJ}{\mathrm{HJ}}
\newcommand{\Z}{\mathbb{Z}}

\DeclareMathOperator{\type}{type}
\newcommand{\ksum}{\kappa_{\mathrm{sum}}}
\newcommand{\kow}{\kappa_{\mathrm{ow}}}
\newcommand{\ksym}{\kappa_{\mathrm{sym}}}
\newcommand{\kall}{\kappa_{\mathrm{all}}}

\title{\bfseries One-Weight Colorings, the Symmetric Class, and\\
Lower Bounds for Hales--Jewett Numbers}
\author{Younes Mouhib\thanks{Department of Mathematics, ETH Z\"urich.
This note extracts and consolidates material from the author's MSc
thesis~\cite{thesis}, written under the supervision of Prof.\ Dr.\ L.\
Halbeisen; the records also appear in~\cite{mh}.}}
\date{July 2026}
 
\begin{document}
\maketitle
 
\begin{abstract}
A coloring of the Hales--Jewett cube $[t]^n$ is \emph{symmetric} if it is
invariant under all coordinate permutations, and \emph{one-weight} if it reads
only an integer-weighted count of the letters. We prove that the two classes
coincide --- a radix weight realizes every symmetric coloring --- so the
symmetric lower-bound problem for the Hales--Jewett numbers is exactly a
one-dimensional coloring problem about homothetic copies of a $t$-point set,
the case $d=1$ of Gallai's theorem. Optimizing the weight yields
$\HJ(3,3)\ge22$ and $\HJ(4,2)\ge14$, the latter in closed form from the new
Gallai homothety numbers $G_2(\{0,2,3,5\})=67$ and $G_2(\{0,1,5,6\})=80$; new
values at three colors --- $G_3(\{0,1,3\})=42$, $G_3(\{0,1,4\})=57$ and
$G_3(\{0,2,5\})\ge77$ --- give $\HJ(3,3)\ge16$ from a one-line
certificate. An anatomy of the $(4,2)$ palette locates the source of its
compression: it is an extremal object of the bracket regime plus a single
boundary scale. An
exhaustive census shows how thin the class is: of the $1644$ line-free
$2$-colorings of $[3]^3$, exactly $36$ are symmetric. For lines with at most
$K$ active coordinates the same machinery gives infinite \emph{bracket}
numbers, $\HJ^{[12]}(3,3)=\HJ^{[12]}(4,2)=\infty$, strictly beyond the
sum-type ceilings $\ksum(3,3)=11$ and $\ksum(4,2)=10$; for lines whose active
set is an interval the machinery is provably blind, the interval ceiling
$\lambda(3,r)$ is settled for every $r$ by assembling the known bounds,
and a SAT computation gives the exact value $\HJ^{(1)}(3)=5>4=\HJ(3)$. We close with the Collapse, diagonal-only, and
symmetric-extremality conjectures and with open problems on optimal weights.
Every certificate displayed in this note has been re-verified by direct
enumeration, independently of any solver.
\end{abstract}
 
 
\section{Introduction}\label{sec:intro}
 
Let $[t]=\{1,\dots,t\}$. A \emph{root} is a word
$\tau\in([t]\cup\{*\})^n\setminus[t]^n$; substituting $a\in[t]$ for every $*$
gives $\tau(a)\in[t]^n$, and the \emph{combinatorial line} generated by $\tau$
is $L_\tau=\{\tau(1),\dots,\tau(t)\}$. The coordinates of $\tau$ carrying $*$
form its \emph{active set} $I(\tau)\subseteq[n]$.
 
\begin{theorem}[Hales--Jewett~\cite{HJ1963}]
For all $t,r\ge1$ there is an $n$ such that every $r$-coloring of $[t]^n$
contains a monochromatic combinatorial line.
\end{theorem}
 
The least such $n$ is the Hales--Jewett number $\HJ(t,r)$, with
$\HJ(t):=\HJ(t,2)$. Exact values are scarce: $\HJ(2,r)=r$
(Proposition~\ref{prop:small-interval}), and $\HJ(3)=4$ by Hindman and
Tressler~\cite{firstnontrivialHJ_is_4}. For $t\ge4$ no value is known. On the
upper side, Shelah's proof~\cite{s.proof_of_HJ} (see
also~\cite{jungic_ramsey_notes,golshanimirabi2021})
gives primitive-recursive bounds, Lavrov~\cite{lavrov} proved
$\HJ(4,2)<10^{11}$, and Conlon~\cite{conlon2021monochromaticcombinatoriallineslength}
proved $\HJ(3,r)\le2^{2^{cr}}$ for an absolute constant $c$. On the lower
side, the classical route is through van der Waerden's
theorem~\cite{vanderWaerden1927}: since the letter sum maps a line with $k$
active coordinates onto a $t$-term arithmetic progression of gap $k$,
\begin{equation}\label{eq:vdw-shadow}
\HJ(t,r)\ \ge\ \Bigl\lceil\tfrac{W(t,r)-1}{t-1}\Bigr\rceil ,
\end{equation}
where $W(t,r)$ is the van der Waerden number. Combined with the bounds of
Kozik--Shabanov~\cite{kozik2016}, of Blankenship--Cummings--Taranchuk
\cite{BLANKENSHIP2018163} ($W(p+1,r)>p^{r-1}2^p$ for primes $p\ge r$,
generalizing Berlekamp~\cite{Berlekamp1968}), this gives
$\HJ(t)=\Omega(2^t/t)$ and, e.g., $\HJ(5,2)\ge45$, $\HJ(6,2)\ge227$ from
$W(5,2)=178$, $W(6,2)=1132$~\cite{LandmanRobertson}.
 
This note isolates a single mechanism behind stronger lower bounds and behind
several structural results around them: the \emph{one-weight coloring}, which
colors a word by an integer-weighted count of its letters. The contributions
are organized as follows.
 
Section~\ref{sec:type} sets up the type map onto the discrete simplex and
quantifies how few symmetric colorings there are, including an exhaustive
census of the critical cube $[3]^3$: of its $1644$ line-free $2$-colorings,
exactly $36$ are symmetric. Section~\ref{sec:ow} develops the one-weight
machinery --- the corner-tuple reduction and the line--pattern correspondence
--- and proves the characterization: \emph{one-weight colorings are exactly
the symmetric colorings} (Theorem~\ref{thm:radix}), so the apparent hierarchy
of weight ranks is flat. Section~\ref{sec:gallai} reads the resulting
lower-bound problem as the one-dimensional Gallai theorem, extracts a
closed-form bound through the Gallai homothety numbers $G_r(S)$, and charts
the ratio landscape at small parameters: weight-independent and tight at
$(3,2)$, and lifted from $13$ to $16$ at $(3,3)$ by new Gallai numbers at
three colors.
Section~\ref{sec:records} proves the records $\HJ(3,3)\ge22$ and
$\HJ(4,2)\ge14$, each by a single weight, dissects the $(4,2)$ palette into
an extremal bracket object plus one boundary scale, and contrasts it with
$(3,3)$, where no comparably compressed witness is known. Section~\ref{sec:bracket} develops
the bracket numbers $\HJ^{[K]}$ and their ceilings and proves
$\HJ^{[12]}(3,3)=\HJ^{[12]}(4,2)=\infty$ by periodic one-weight palettes,
strictly beyond the sum-type ceilings $\ksum(3,3)=11$, $\ksum(4,2)=10$.
Section~\ref{sec:interval} treats the interval numbers $\HJ^{(q)}$, to which
the entire symmetric machinery is provably blind, assembles the literature
into the exact ceiling $\lambda(3,r)$ for every $r$
(Proposition~\ref{prop:lambda3}), and establishes $\HJ^{(1)}(3)=5$. Section~\ref{sec:open} collects the conjectures and open
problems: Collapse, diagonal-only colorings, symmetric extremality, and the
optimal-weight questions.
 
\emph{Verification.} Statements marked \emph{computer-assisted} rest on finite
computations. Every certificate displayed in this note --- the census of
Table~\ref{tab:census}, the palettes of Theorems~\ref{thm:hj42-14},
\ref{thm:hj12-33} and~\ref{thm:hj12-42}, the palette of
Proposition~\ref{prop:kstar}, the power-residue instances of
Theorem~\ref{thm:character} at $p=11,37,97,139$, the Gallai numbers of
Table~\ref{tab:gallai}, the interval certificate of
Theorem~\ref{thm:hj33-16}, the periodic palette of
Proposition~\ref{prop:hj33-15}, and the witness underlying
Theorem~\ref{thm:hj13}
--- has been re-verified for this note by direct enumeration from the
definitions, independently of any solver. The one certificate too large to
display, the $253$-cell witness of Theorem~\ref{thm:hj33-22}, is recorded
in~\cite{thesis} together with a dependency-free verifier. Avoidance
certificates are thus verified by direct enumeration; \emph{unsatisfiability}
(forcing) claims necessarily rest on SAT solvers; where stated, the number
of independent solvers used is given.
 
\section{The Type Map and the Scarcity of Symmetric Colorings}\label{sec:type}
 
Write $\sigma(w)=\sum_{i=1}^n w_i$ for the \emph{weight} of $w\in[t]^n$, and
\[
\type(w)=(a_1,\dots,a_t),\qquad a_j=|\{i: w_i=j\}| ,
\]
for its \emph{type}. Since $\sum_j a_j=n$, the type ranges over the discrete
simplex
\[
T_n^{(t)}=\Bigl\{a\in\Z_{\ge0}^{t}:\ \textstyle\sum_{j}a_j=n\Bigr\},
\qquad
\bigl|T_n^{(t)}\bigr|=\binom{n+t-1}{t-1},
\]
of size polynomial in $n$, against the $t^n$ words of the cube. The symmetric
group $S_n$ acts on $[t]^n$ by permuting coordinates; a coloring is
\emph{symmetric} if it is invariant under this action.
 
\begin{lemma}\label{lem:sym}
The $S_n$-orbits of $[t]^n$ are exactly the fibers of $\type$. Consequently a
coloring $c$ is symmetric if and only if it factors as
$c=\bar c\circ\type$ through a unique $\bar c\colon T_n^{(t)}\to[r]$, called
its \emph{descent}.
\end{lemma}
 
\begin{proof}
A coordinate permutation preserves the multiset of letters, so $\type$ is
constant on orbits. Conversely, if $\type(w)=\type(w')$ the level sets
$\{i:w_i=j\}$ are equinumerous for each $j$, and any permutation matching them
carries $w$ to $w'$. Symmetry is thus constancy on fibers, i.e.\
factorization through $\type$; uniqueness holds because $\type$ is onto,
realized by $1^{a_1}\cdots t^{a_t}$.
\end{proof}
 
Symmetric colorings are therefore rare in a strong quantitative sense: there
are $r^{\binom{n+t-1}{t-1}}$ of them among the $r^{t^n}$ colorings of the
cube, a fraction $r^{\binom{n+t-1}{t-1}-t^n}$ that vanishes doubly
exponentially in $n$. The scarcity persists among the extremal colorings. At
$(t,r)=(3,2)$ the cube $[3]^3$ is the largest admitting a line-free
$2$-coloring ($\HJ(3,2)=4$~\cite{firstnontrivialHJ_is_4}), so its line-free
colorings admit a complete account.
 
\begin{theorem}[{Census of $[3]^3$; computer-verified}]\label{thm:census}
Exactly $1644$ of the $2^{27}$ two-colorings of $[3]^3$ are line-free. Their
distribution by coordinate stabilizer, and the subdivision of the symmetric
class by realizing weight \textup{(}Definition~\ref{def:ow}\textup{)}, is
given in Table~\ref{tab:census}. In particular only $36$ --- about one in $46$ ---
are symmetric.
\end{theorem}
 
\begin{table}[ht]
\centering
\caption{The $1644$ line-free $2$-colorings of $[3]^3$, by coordinate
stabilizer in $S_3$ and by weight. Orbit counts follow from
orbit--stabilizer. All counts were obtained by exhaustive enumeration of the
$2^{27}$ colorings against the $37$ lines and re-verified independently; the
symmetric total was cross-checked on the ten-cell simplex $T_3^{(3)}$ via
Lemma~\ref{lem:reduction}.}
\label{tab:census}
\begin{tabular}{@{}llrr@{}}
\toprule
class & stabilizer / weight & number & orbits\\
\midrule
symmetric ($=$ one-weight) & $S_3$ & $36$ & $36$\\
\quad sum-type & weight $(1,2,3)$ & $16$ & $16$\\
\quad one-weight, non-arithmetic & e.g.\ $(0,1,3)$, $(0,2,3)$ & $20$ & $20$\\
\addlinespace
non-symmetric & $\subsetneq S_3$ & $1608$ & $360$\\
\quad block-symmetric & $C_2$ (one transposition) & $504$ & $168$\\
\quad cyclic & $C_3$ & $24$ & $12$\\
\quad asymmetric & trivial & $1080$ & $180$\\
\midrule
total & & $1644$ & $396$\\
\bottomrule
\end{tabular}
\end{table}

\begin{remark}[The rainbow dual measures the same scarcity]\label{rem:rainbow}
A symmetric coloring is constant on the $\binom{n+t-1}{t-1}$ type classes, so
it uses at most polynomially many colors; but the anti--Hales--Jewett
threshold --- the least $r$ forcing a rainbow line under every surjective
$r$-coloring --- exceeds $(t-1)^n$~\cite{zheng2024rainbow}. Hence for $t\ge3$
the largest symmetric rainbow-free coloring falls short of the truth by an
exponential factor: symmetry, which Section~\ref{sec:records} shows is
(conjecturally) the tight model for the monochromatic problem, is
exponentially lossy for the rainbow dual. (An explicit rainbow-free
$24$-coloring of $[3]^4$, re-verified against all $175$ lines, improves the
lower bound of~\cite{zheng2024rainbow} to $\mathrm{ah}(3,4)\ge25$;
see~\cite{mh}. The rainbow twin of the reduction~\eqref{eq:vdw-shadow} runs
to anti--van der Waerden numbers~\cite{berikkyzy2017antivdw}.)
\end{remark}
 
\section{One-Weight Colorings Are the Symmetric Colorings}\label{sec:ow}
 
\begin{definition}\label{def:ow}
A \emph{one-weight coloring} of $[t]^n$ is
\[
c_{\omega,\chi}(w)=\chi\bigl(\langle\omega,\type(w)\rangle\bigr),
\qquad
\langle\omega,a\rangle=\textstyle\sum_{j}\omega_j a_j ,
\]
given by a \emph{weight} $\omega\in\Z^t$ and a \emph{palette}
$\chi\colon\Z\to[r]$ (taken $m$-periodic when so indicated). The
\emph{sum-type} colorings are the case $\omega=(1,2,\dots,t)$, where
$\langle\omega,\type(w)\rangle=\sigma(w)$.
\end{definition}
 
\begin{remark}\label{rem:affine}
Weights $\omega$ and $\alpha\omega+\beta\mathbf1$ ($\alpha\ne0$) define the
same family, since
$\langle\alpha\omega+\beta\mathbf1,\type(w)\rangle
=\alpha\langle\omega,\type(w)\rangle+\beta n$ is an affine reparametrization
absorbed into $\chi$. Thus $\omega$ matters only up to scaling and
translation; in particular at $t=2$ every nonconstant weight is affinely
$(1,2)$, so one-weight, sum-type and (by Theorem~\ref{thm:radix}) symmetric
all coincide there.
\end{remark}
 
Every one-weight coloring factors through $\type$, hence is symmetric, and
sum-type colorings are one-weight, giving classes
\[
\mathrm{sum}\ \subseteq\ \mathrm{ow}\ \subseteq\ \mathrm{sym}\ \subseteq\
\mathrm{all},
\]
the last inclusion strict for $n,r\ge2$.
For a class $\mathcal C$ of colorings, $\HJ_{\mathcal C}(t,r)$ denotes the
least $n$ such that \emph{every} $\mathcal C$-coloring of $[t]^n$ has a
monochromatic line; shrinking $\mathcal C$ relaxes the quantifier, so
$\HJ_{\mathrm{sum}}\le\HJ_{\mathrm{ow}}\le\HJ_{\mathrm{sym}}
\le\HJ_{\mathrm{all}}=\HJ(t,r)$.
 
The mechanism underlying everything below is a pair of reductions. For
$1\le k\le n$ and $v\in T^{(t)}_{n-k}$ write
\[
C_{k,v}=\{v+ke_1,\dots,v+ke_t\}\subseteq T_n^{(t)}
\]
for the \emph{corner tuple} of scale $k$ and base $v$ ($e_a$ the standard
basis of $\Z^t$), and, for $K\ge1$, write $L^{[K]}([t]^n)$ for the set of
lines with at most $K$ active coordinates.
 
\begin{lemma}[Symmetric reduction]\label{lem:reduction}
Let $c=\bar c\circ\type$ be symmetric (Lemma~\ref{lem:sym}) and let $\tau$ be a root with
$k=|I(\tau)|$ active coordinates and inactive type $v\in T^{(t)}_{n-k}$. Then
\begin{equation}\label{eq:type-identity}
\type(\tau(a))=v+ke_a\qquad(a\in[t]),
\end{equation}
so $L_\tau$ is monochromatic under $c$ iff $C_{k,v}$ is monochromatic under
$\bar c$. For every $1\le K\le n$ the following are equivalent:
\textup{(i)} no line of $L^{[K]}([t]^n)$ is monochromatic under $c$;
\textup{(ii)} no corner tuple $C_{k,v}$ with $1\le k\le K$ is monochromatic
under $\bar c$. At $K=n$: $c$ is line-free iff no corner tuple is
monochromatic.
\end{lemma}
 
\begin{proof}
In $\tau(a)$ the $k$ active coordinates carry the letter $a$ and the inactive
ones contribute $v$, giving~\eqref{eq:type-identity}. Every line of $L^{[K]}$
realizes some pair $(k,v)$ with $k\le K$, and conversely every such pair is
realized by the root ${*}^{k}1^{v_1}\cdots t^{v_t}$.
\end{proof}
 
\begin{remark}[Compression]\label{rem:compression}
The number of corner tuples with $1\le k\le K$ is
$\sum_{k=1}^{K}\binom{n-k+t-1}{t-1}$, equal to $\binom{n+t-1}{t}$ at $K=n$.
Thus the ${\approx}1.3\times10^{8}$ points and ${\approx}1.7\times10^{10}$
lines of $[3]^{17}$ compress to $171$ cells and $969$ tuples, and the
$16{,}777{,}216$ points of $[4]^{12}$ to $455$ cells and $1365$ tuples. A
lower-bound witness is published in \cite{mh} as an explicit table on the cells; its proof
is Lemma~\ref{lem:reduction} together with a finite check of the corner
tuples, independent of any solver.
\end{remark}
 
For one-weight colorings the corner condition collapses by one further
dimension, onto $\Z$ itself. Write $S_\omega=\{\omega_1,\dots,\omega_t\}$.
 
\begin{lemma}[Line--pattern correspondence]\label{lem:pattern}
Under $c_{\omega,\chi}$, a line with $k$ active coordinates and inactive type
$v$, with $b:=\langle\omega,v\rangle$, carries the levels
\[
H_\omega(k,b)\ =\ b+k\,S_\omega\ \subseteq\ \Z,
\]
a scale-$k$ homothet of the weight set based at $b$. Consequently
$c_{\omega,\chi}$ has no monochromatic line in $L^{[K]}([t]^n)$ iff $\chi$ has
no monochromatic homothet $b+kS_\omega$ with $1\le k\le K$ and
$b\in\langle\omega,T^{(t)}_{n-k}\rangle$.
\end{lemma}
 
\begin{proof}
By~\eqref{eq:type-identity} the $a$-th word has level
$\langle\omega,v+ke_a\rangle=b+k\omega_a$; a one-weight coloring is constant
on $\langle\omega,\cdot\rangle$-fibers, so the line is monochromatic iff
$\chi$ is constant on $H_\omega(k,b)$.
\end{proof}
 
One might expect the one-weight colorings to be a strict, simpler subclass of
the symmetric ones. In a fixed dimension they are not: allowing the weight to
grow with $n$ recovers the entire symmetric class.
 
\begin{theorem}[One-weight $=$ symmetric; the rank hierarchy is
flat]\label{thm:radix}
Fix $t\ge2$, $r\ge2$, $n\ge1$. The colorings of $[t]^n$ of the form
$c_{\omega,\chi}$ are exactly the symmetric colorings. Specifically, for the
\emph{radix weight}
\[
\omega^\ast=(0,1,p,p^2,\dots,p^{t-2})\qquad\text{with an arbitrary integer } p>n
\]
(primality plays no role),
the functional $\langle\omega^\ast,\cdot\rangle$ is injective on $T_n^{(t)}$
and carries each corner tuple $C_{k,v}$ bijectively onto a carry-free lattice
corner; consequently every symmetric coloring $\bar c$ is realized as
$c_{\omega^\ast,\chi}$ with $\chi(\langle\omega^\ast,v\rangle)=\bar c(v)$, and
$c_{\omega^\ast,\chi}$ is line-free iff no corner tuple is monochromatic
under $\bar c$ (Lemma~\ref{lem:reduction}). Moreover any coloring
built from finitely many weights
$w\mapsto\Psi(\langle\omega^{(1)},\type(w)\rangle,\dots,
\langle\omega^{(s)},\type(w)\rangle)$ is symmetric, hence one-weight: there
is no multi-weight tier, and
\[
\HJ_{\mathrm{ow}}(t,r)=\HJ_{\mathrm{sym}}(t,r).
\]
\end{theorem}
 
\begin{proof}
On $T_n^{(t)}$ each entry satisfies $0\le v_a\le n<p$ for $a\ge2$, so
$\langle\omega^\ast,v\rangle=\sum_{a\ge2}v_a\,p^{a-2}$ is the base-$p$ integer
with digit string $(v_2,\dots,v_t)$; it is injective because $v_1$ is
determined by $v_1=n-\sum_{a\ge2}v_a$. A corner point $v+ke_a$ with $a\ge2$
has $a$-digit $v_a+k\le n<p$ (as $v_a\le n-k$), so no carry occurs and
$\langle\omega^\ast,\cdot\rangle$ maps $C_{k,v}$ bijectively onto the digit
corner $\{(v_2,\dots,v_t)+k\hat e_a\}_{a\in[t]}$, where $\hat e_a$ is the
$a$-th standard basis vector of the digit space for $a\ge2$ and
$\hat e_1:=0$ (the point $v+ke_1$ keeps the digit string of $v$, since
$\omega^\ast_1=0$). Hence
$\chi(\langle\omega^\ast,v\rangle):=\bar c(v)$ is well defined (extended
arbitrarily off the image), $c_{\omega^\ast,\chi}=\bar c\circ\type$, and
line-freeness transfers by Lemma~\ref{lem:reduction}. Every $c_{\omega,\chi}$
is symmetric; a multi-weight coloring factors through $\type$, hence is
symmetric, hence one-weight by the above. Since the two classes coincide in
every dimension, the defining conditions of $\HJ_{\mathrm{ow}}$ and
$\HJ_{\mathrm{sym}}$ agree for each $n$, so the thresholds are equal.
\end{proof}
 
\begin{remark}[Finite reach versus all dimensions]\label{rem:caveat}
The radix weight depends on $n$ through $p>n$. Where a \emph{fixed} weight
and palette must avoid lines in \emph{every} dimension --- the bracket regime
$\HJ^{[K]}=\infty$ of Section~\ref{sec:bracket} --- the device is
unavailable, and there the one-weight class is a genuine restriction.
Theorem~\ref{thm:radix} is a statement about finite reach only. The
dependence on $n$ is moreover not an artifact of the proof:
$\langle\omega,\cdot\rangle$ maps $T_n^{(t)}$ into an interval of
$D_\omega n+1$ integers, $D_\omega=\max\omega-\min\omega$, so injectivity
forces $D_\omega\ge\bigl(\binom{n+t-1}{t-1}-1\bigr)/n=\Theta(n^{t-2})$ for
$t\ge3$; the radix diameter $p^{t-2}$ is optimal up to a factor $(t-1)!$,
and realizing the full symmetric class in dimension $n$ \emph{requires} a
weight of diameter polynomial in $n$.
\end{remark}
 
\section{The Gallai Reading and a Closed-Form Bound}\label{sec:gallai}
 
For a nonconstant weight $\omega$ define its \emph{reach} $n_\omega(t,r)$
as the largest $n\ge0$ for which some palette $\chi$ makes
$c_{\omega,\chi}$ line-free on $[t]^n$; by Lemma~\ref{lem:pattern} and
Hales--Jewett it is finite, and by Theorem~\ref{thm:radix} its optimum over
weights is $\HJ_{\mathrm{sym}}(t,r)-1$. Lemma~\ref{lem:pattern} identifies the
finiteness of $n_\omega$ with a classical statement. Recall
\emph{Gallai's theorem} in dimension one --- due to Gallai, first published by
Rado and independently by Witt; see~\cite{hans,LandmanRobertson} for
expositions: for every finite $S\subset\Z$ and every finite coloring of
$\Z$ there exist $b\in\Z$ and $k\ge1$ with $b+kS$ monochromatic.
 
\begin{theorem}[Gallai form of the one-weight bound]\label{thm:gallai-form}
Let $\omega\in\Z^t$ have value set $S=S_\omega$, $|S|=t$. For each $n$,
$c_{\omega,\chi}$ is line-free on $[t]^n$ iff $\chi$ has no monochromatic
homothet $b+kS$ with $1\le k\le n$ and
$b\in\langle\omega,T^{(t)}_{n-k}\rangle$. As $n\to\infty$ the scale $k$
ranges over all of $\Z_{\ge1}$ and, after the affine normalization
$\min S=0$ of Remark~\ref{rem:affine}, the base $b$ over the additive monoid
generated by $S$; that no finite coloring of $\Z$ avoids all such homothets
is precisely the case $d=1$ of Gallai's theorem with bases restricted to the
monoid --- a restriction that loses nothing --- and
$n_\omega(t,r)<\infty$ is that theorem read inside the cube.
\end{theorem}
 
\begin{proof}
The equivalence is Lemma~\ref{lem:pattern} at $K=n$. For the limit: with
$\min S=0$, any element $m$ of the monoid generated by $S$ is
$\sum_{a}c_a\omega_a$ with $c_a\in\Z_{\ge0}$, realized as
$\langle\omega,v\rangle$ for $v\in T^{(t)}_{n-k}$ once $n-k\ge\sum_a c_a$,
padding the coordinate of weight $0$. Restricting the base to the monoid
loses nothing: let $g=\gcd(S\setminus\{0\})$ and $S=gS'$; the monoid
generated by $S'$ contains every integer beyond its Frobenius number
$N_0$. Given $\chi$, van der Waerden applied to the coloring
$x\mapsto\chi(gx)$ produces a monochromatic $(\max S'+1)$-term progression
inside the window $[N_0,\,N_0+W(\max S'+1,r)-1]$, hence a monochromatic
$b'+kS'$ with $b'\ge N_0$ in that monoid; then $gb'$ lies in the monoid
generated by $S$ and $gb'+kS=g(b'+kS')$ is monochromatic under $\chi$. The
converse implication is trivial, monoid-based homothets being homothets.
\end{proof}
 
The full-interval form of this avoidance yields a closed-form bound. Define
the \emph{Gallai homothety number} $G_r(S)$ as the least $N$ such that every
$r$-coloring of $N$ consecutive integers contains a monochromatic homothet
$b+kS$ ($k\ge1$) lying inside them. Normalizing $\min S=0$ and writing
$D=\max S$ for the diameter, $G_r(S)\le W(D+1,r)<\infty$, since a
monochromatic $(D+1)$-term progression $\{c,c+k,\dots,c+Dk\}$ contains
$c+kS$.
 
\begin{theorem}[Gallai-number lower bound]\label{thm:gallai-bound}
For every $S\subset\Z$ with $|S|=t$ and diameter $D$,
\[
\HJ(t,r)\ \ge\ \Bigl\lceil\tfrac{G_r(S)-1}{D}\Bigr\rceil,
\qquad\text{hence}\qquad
\HJ(t,r)\ \ge\ \max_{|S|=t}\Bigl\lceil\tfrac{G_r(S)-1}{D_S}\Bigr\rceil .
\]
\end{theorem}
 
\begin{proof}
Normalize $\min S=0$ and take $\omega$ with value set $S$. The functional
$\pi_\omega(w)=\langle\omega,\type(w)\rangle$ maps $[t]^n$ into the interval
$I_n=[0,Dn]$ of $Dn+1$ integers and, by Theorem~\ref{thm:gallai-form}, every
line onto a homothet of $S$ inside $I_n$. If $Dn+1\le G_r(S)-1$, some
$r$-coloring $\xi$ of $I_n$ contains no monochromatic homothet of $S$ at all;
then $c=\xi\circ\pi_\omega$ is line-free and $\HJ(t,r)>n$. The largest such
$n$ is $\lfloor(G_r(S)-2)/D\rfloor$, and
$\lfloor(G_r(S)-2)/D\rfloor+1=\lceil(G_r(S)-1)/D\rceil$.
\end{proof}
 
\begin{corollary}[Arithmetic weight $=$ van der Waerden]\label{cor:arith}
$G_r(\{0,1,\dots,t-1\})=W(t,r)$, and for the arithmetic weight
$\omega=(0,1,\dots,t-1)$ the bound of Theorem~\ref{thm:gallai-bound} is
exactly~\eqref{eq:vdw-shadow}; moreover here the closed
form is exact: $n_\omega(t,r)=\lfloor(W(t,r)-2)/(t-1)\rfloor$, i.e.\ the
bound $\lceil(W(t,r)-1)/(t-1)\rceil$ equals $n_\omega+1$.
\end{corollary}
 
\begin{proof}
Homothets of $\{0,\dots,t-1\}$ are exactly $t$-term progressions, giving the
first claim and the bound. For the reach: the bases realized in the cube are
all of $[0,(t-1)(n-k)]$ (weights $0,\dots,t-1$ realize every intermediate
value), so the homothets realized are exactly the $t$-term progressions in
$[0,(t-1)n]$, and a progression-free $r$-coloring of that interval exists iff
$(t-1)n+1\le W(t,r)-1$.
\end{proof}
 
\begin{lemma}[Affine covariance]\label{lem:affine-G}
For $c\ge1$ and $d\in\Z$, $G_r(cS+d)=c\,(G_r(S)-1)+1$, and
$G_r(-S)=G_r(S)$. Hence $\lceil(G_r(S)-1)/D_S\rceil$ is invariant under
$S\mapsto\pm cS+d$, and in the maximum of Theorem~\ref{thm:gallai-bound}
one may assume $\min S=0$ and $\gcd S=1$. (For $|S|=3$ the scaling
identity is Theorem~1 of Brown--Landman--Mishna~\cite{BrownLandmanMishna1997}.)
\end{lemma}
 
\begin{proof}
Translation and reflection act on colorings. For the scaling, split a
window of $N$ consecutive integers into its $c$ residue classes: every
homothet of $cS$ lies in a single class, each class is order-isomorphic to
a window of $\lceil N/c\rceil$ or $\lfloor N/c\rfloor$ consecutive
integers on which the homothets of $cS$ correspond to the homothets of
$S$, and the classes are colored independently. Avoidance is therefore
possible iff $\lceil N/c\rceil\le G_r(S)-1$, i.e.\ iff
$N\le c\,(G_r(S)-1)$.
\end{proof}
 
\begin{table}[ht]
\centering
\caption{The Gallai-number bound from a single input. The values $42$, $57$,
$67$, $80$ and the bound $G_3(\{0,2,5\})\ge77$ are new. Avoidance colorings
at $G_r-1$ (at $76$ for $\{0,2,5\}$) were re-verified by direct enumeration;
forcing at $G_r$ was established by SAT --- two independent solvers for
$27$, $35$, $67$, $80$, three for $42$ and $57$. The rows $\{0,2,3,5\}$ and
$\{0,1,5,6\}$ give the $(4,2)$ record of Theorem~\ref{thm:hj42-14}; the
rows at three colors give Theorem~\ref{thm:hj33-16}.}
\label{tab:gallai}
\begin{tabular}{@{}lllll@{}}
\toprule
$(t,r)$ & $S$ & $D$ & $G_r(S)$ & $\lceil(G_r(S)-1)/D\rceil$\\
\midrule
$(2,r)$ & $\{0,1\}$     & $1$ & $r+1$         & $r$ \;(tight)\\
$(3,2)$ & $\{0,1,2\}$   & $2$ & $9=W(3,2)$    & $4$ \;(tight)\\
$(3,3)$ & $\{0,1,2\}$   & $2$ & $27=W(3,3)$   & $13$\\
$(3,3)$ & $\{0,1,3\}$   & $3$ & $42$          & $14$\\
$(3,3)$ & $\{0,1,4\}$   & $4$ & $57$          & $14$\\
$(3,3)$ & $\{0,2,5\}$   & $5$ & $\ge77$       & $\ge\mathbf{16}$\\
$(4,2)$ & $\{0,1,2,3\}$ & $3$ & $35=W(4,2)$   & $12$\\
$(4,2)$ & $\{0,2,3,5\}$ & $5$ & $67$          & $\mathbf{14}$\\
$(4,2)$ & $\{0,1,5,6\}$ & $6$ & $80$          & $\mathbf{14}$\\
\bottomrule
\end{tabular}
\end{table}
 
\begin{remark}[The closed form against the true reach]\label{rem:closed-vs-reach}
Theorem~\ref{thm:gallai-bound} counts \emph{all} homothets in the interval,
including those whose base lies outside the monoid generated by $S$; the cube
realizes only the monoid bases, so the dimension bound $n_\omega+1$ is at
least the closed-form value, with equality when no relevant base is
missing. The
two coincide for arithmetic $S$ (Corollary~\ref{cor:arith}) and for the
$(4,2)$ record below, but need not coincide: for the radix weight
$\omega=(0,1,29)$ the exhibited reach is $n_\omega\ge21$
(Theorem~\ref{thm:hj33-22}) while $G_3(\{0,1,29\})$, hence its closed form,
has not been certified; where the two diverge, the gap is
precisely the carry homothets that never occur in the cube for $n<29$
(Theorem~\ref{thm:radix}). The closed form is a cheap certificate; the
symmetric reach is the true power. Read backwards,
Theorem~\ref{thm:gallai-bound} gives $G_r(S)\le D\cdot\HJ(t,r)+1$, sharp at
$t=2$ where $G_r(\{0,D\})=rD+1$ (Lemma~\ref{lem:affine-G}), and the optimum
$\max_{|S|=t}\lceil(G_r(S)-1)/D_S\rceil$ becomes an extremal problem on
Gallai numbers per unit diameter (Question~\ref{q:weights}). By
Lemma~\ref{lem:affine-G} the ratio depends only on the affine class of
$S$: e.g.\ $G_2(\{0,2,4,6\})=2\cdot34+1=69$ exceeds $67$, yet
$\lceil68/6\rceil=12$ merely reproduces the arithmetic row of
Table~\ref{tab:gallai}.
\end{remark}
 
\subsection{New Gallai numbers at three colors, and $\HJ(3,3)\ge16$}\label{sec:landscape}
 
At $(2,r)$ the closed form is tight for \emph{every} weight: all two-point
sets are affinely $\{0,1\}$ (Lemma~\ref{lem:affine-G}) and
$G_r(\{0,1\})=r+1$ by pigeonhole, so the ratio is $r=\HJ(2,r)$ throughout.
The first nontrivial pair behaves the same way.
 
\begin{proposition}[{Flatness at $(3,2)$; after
Brown--Landman--Mishna~\cite{BrownLandmanMishna1997} and
Kim--Rho~\cite{KimRho2012}}]\label{prop:flat32}
Let $S=\{0,s,s+u\}$ with $\gcd(s,u)=1$ and diameter $D=s+u$. Then
\[
G_2(S)\;=\;
\begin{cases}
4D, & \{s,u\}=\{1,4m\}\ \text{for some } m\ge1,\\[2pt]
4D+1, & \text{otherwise},
\end{cases}
\]
and consequently, for \emph{every} three-point $S$,
\[
\Bigl\lceil\tfrac{G_2(S)-1}{D_S}\Bigr\rceil\;=\;4\;=\;\HJ(3,2):
\]
at $(3,2)$ the closed form of Theorem~\ref{thm:gallai-bound} is
weight-independent and tight.
\end{proposition}
 
\begin{proof}
Since a copy $b+k\{1,1+s,1+s+u\}$ equals $(b+k)+kS$, homothety-avoidance
on an interval is the same for the two sets, so $G_2(S)$ is the number
$f(s,u)$ of Brown--Landman--Mishna: the least $N$ such that every
$2$-coloring of $[1,N]$ contains a monochromatic homothetic copy of
$\{1,1+s,1+s+u\}$. They proved $f(s,u)\le4(s+u)+1$, with equality
whenever $s/g\not\equiv0$ and $u/g\not\equiv0\pmod4$ ($g=\gcd(s,u)$) and
in further cases, leaving for the family $\{t,4mt\}$ two candidate
values~\cite{BrownLandmanMishna1997}; Kim and Rho settled it:
$f(4mt,t)=f(t,4mt)=4(4mt+t)-t+1$, and $f(s,u)=4(s+u)+1$ for every other
pair~\cite[Theorem~12]{KimRho2012}. For primitive $S$ the exceptional
pairs force $t=1$, since $\gcd(t,4mt)=t$, giving the displayed
dichotomy; general $S$ reduces to this by Lemma~\ref{lem:affine-G}. For
the ratio, $(G_2(S)-1)/D$ equals $4$ in the generic case and $4-1/D$ in
the exceptional one, so its ceiling is $4$ in both; and
$\HJ(3,2)=4$~\cite{firstnontrivialHJ_is_4} makes the bound tight. As an
independent check, all fourteen primitive classes with $D\le9$ were
recomputed by SAT for this note, the avoidance certificates re-verified
by direct enumeration from the definitions; every value matches the
formula, the two exceptional classes in range being
$G_2(\{0,1,5\})=20$ and $G_2(\{0,1,9\})=36$ ($m=1,2$).
\end{proof}
 
Combined with the $t=2$ row of Table~\ref{tab:gallai}, where
$G_r(\{0,1\})=r+1$ gives the ratio $r=\HJ(2,r)$ for every weight, the
optimized closed form is tight --- for \emph{every} weight, not only the
arithmetic one --- at every pair $(t,r)$ where $\HJ$ is known exactly.
This strengthens conclusion \emph{(i)} of Remark~\ref{rem:records} from
the van der Waerden shadow to the whole one-weight class, and makes
$(3,3)$ the first pair where the optimum
$\max_S\lceil(G_r(S)-1)/D_S\rceil$ is genuinely open.
 
\begin{theorem}[$\HJ(3,3)\ge16$ in closed form; computer-assisted]\label{thm:hj33-16}
$G_3(\{0,1,3\})=42$ and $G_3(\{0,1,4\})=57$ (avoidance certificates
re-verified by direct enumeration; forcing confirmed by three independent
solvers), and
\[
G_3(\{0,2,5\})\ \ge\ 77 .
\]
The latter is witnessed by the explicit $3$-coloring $\xi$ of $\{0,\dots,75\}$
\begin{center}\ttfamily
10201110020122122201020111002012212210\\
10201110020122122001020111002012212200
\end{center}
\textup{(}listing $\xi(0),\dots,\xi(75)$\textup{)}, re-verified by direct
enumeration to contain no monochromatic homothet of $\{0,2,5\}$. Since
$\lceil(G_3(\{0,2,5\})-1)/5\rceil\ge\lceil76/5\rceil=16$,
Theorem~\ref{thm:gallai-bound} gives
\[
\HJ(3,3)\ \ge\ 16 .
\]
\end{theorem}
 
\begin{remark}[Reading the new values]\label{rem:landscape}
The closed form $16$ exceeds both the arithmetic row ($13$, the van der
Waerden shadow) and the previous published record
$\HJ(3,3)\ge14$~\cite{nayda}: a one-line, hand-checkable certificate already
beats the prior state of the art, though the radix reach of
Theorem~\ref{thm:hj33-22} remains far ahead. At $r=2$ the ratio is flat at
$4$ (Proposition~\ref{prop:flat32}); the computed values at $r=3$
($13,14,14,\ge16$ at $D=2,3,4,5$) leave its behavior in the diameter open
(Question~\ref{q:weights}\emph{(i)}). Since the closed form never exceeds
$n_\omega+1\le\HJ_{\mathrm{sym}}$ (Remark~\ref{rem:closed-vs-reach}), a
ratio reaching $22$ would tie --- never beat --- the symmetric frontier, but
would replace the $253$-cell table of Theorem~\ref{thm:hj33-22} by a
one-dimensional certificate.
\end{remark}
 
\section{The Records}\label{sec:records}
 
\subsection{$\HJ(3,3)\ge22$ by a single weight}
 
The previous record was $\HJ(3,3)\ge14$~\cite{nayda}; the van der Waerden
shadow~\eqref{eq:vdw-shadow} gives $13$. By Theorem~\ref{thm:radix} a single
weight reaches the full symmetric value, and $29>21$ makes $(0,1,29)$ the
clean choice.
 
\begin{theorem}[One-weight; computer-assisted]\label{thm:hj33-22}
The weight $\omega=(0,1,29)$ admits a palette $\chi\colon\Z\to\{0,1,2\}$ for
which $c_{\omega,\chi}$ has no monochromatic line on $[3]^{21}$. Hence
\[
\HJ(3,3)\ \ge\ 22 .
\]
\end{theorem}
 
\begin{proof}
Since $29>21$, Theorem~\ref{thm:radix} makes
$\langle\omega,\cdot\rangle\colon(v_1,v_2,v_3)\mapsto v_2+29v_3$ injective on
$T^{(3)}_{21}$, carrying each corner triple $C_{k,v}$ onto the carry-free
planar corner $\{(v_2,v_3),(v_2{+}k,v_3),(v_2,v_3{+}k)\}$. By
Lemma~\ref{lem:pattern}, $c_{\omega,\chi}$ is line-free iff
$\bar c(v):=\chi(v_2+29v_3)$ leaves all such corners non-monochromatic. A
$3$-coloring of the $253$ cells of $T^{(3)}_{21}$ avoiding all
$1771=\binom{23}{3}$ corner triples ($1\le k\le21$) exists: the certificate,
recorded in~\cite{thesis} as the palette $\chi$ on the $253$ realized levels
$\{v_2+29v_3:v_2+v_3\le21\}\subset\{0,\dots,609\}$, is checked by a
dependency-free verifier that re-derives every corner triple from the
definition and reports zero monochromatic triples, cross-checked by two
independent solvers.
\end{proof}
 
\begin{remark}\label{rem:hj33}
The instance ($253$ cells, $1771$ triples) is small yet sits at the
satisfiability phase transition: solvers settle $n=21$ in seconds and stall
at $n=22$, where no refutation has been found, so the bound may still
improve (Section~\ref{sec:open}).
\end{remark}
 
\begin{proposition}[A compact periodic witness; computer-verified]\label{prop:hj33-15}
For the weight $\omega=(0,1,4)$ and the $49$-periodic palette
\begin{center}\ttfamily 0112100212222001110021101201102002110220010200212 \end{center}
\textup{(}listing $\psi(0),\dots,\psi(48)$\textup{)}, the coloring
$c(w)=\psi\bigl(\langle\omega,\type(w)\rangle\bmod49\bigr)$ has no
monochromatic line on $[3]^{14}$, verified by direct enumeration of all
corner triples. Hence $\HJ(3,3)\ge15$ from a $49$-cell certificate, against
the $120$ simplex cells of $T^{(3)}_{14}$.
\end{proposition}
 
\begin{remark}[The compressibility gap at $(3,3)$]\label{rem:gap33}
A (non-exhaustive) search over weights $(0,s,s{+}u)$ with $s\le6$,
$s{+}u\le12$ and moduli $m\le64$, with the bases restricted to those
realized in the cube, found no periodic palette line-free beyond $n=14$.
The anatomy of the $(4,2)$ record (Remark~\ref{rem:anatomy} below) suggests
why: the reach needed at $(3,3)$ is $21$, nine scales above the periodic
ceiling $K=12$ of Section~\ref{sec:bracket}, while the thinning of the
realized bases near the boundary bought only about two scales at these
moduli --- against exactly one scale needed at $(4,2)$. The resulting
ladder of witnesses at $(3,3)$, cheapest to strongest: arithmetic closed
form, $13$; the triples $\{0,1,3\}$ and $\{0,1,4\}$, $14$; the $49$-cell
periodic palette, $15$; the length-$76$ interval certificate, $16$
(Theorem~\ref{thm:hj33-16}); the $253$-cell radix table, $22$.
\end{remark}
 
\subsection{$\HJ(4,2)\ge14$ in closed form}
 
\begin{theorem}[One-weight closed form; computer-verified]\label{thm:hj42-14}
Let $\omega=(0,2,3,5)$ and let $\chi\colon\Z_{26}\to\{0,1\}$ be
\[
\chi=(1,0,1,0,0,1,1,1,1,0,0,1,0,0,0,1,0,0,1,1,1,1,0,0,1,0)
\]
\textup{(}listing $\chi(0),\dots,\chi(25)$\textup{)}. Then
$c(w)=\chi\bigl(\langle\omega,\type(w)\rangle\bmod26\bigr)$ has no
monochromatic line on $[4]^{13}$; hence
\[
\HJ(4,2)\ \ge\ 14 .
\]
The reach of this palette is sharp: at $n=14$ it leaves exactly three
monochromatic corner quadruples.
\end{theorem}
 
\begin{proof}
The coloring is symmetric, so by Lemma~\ref{lem:reduction} ($K=n=13$)
line-freeness is the non-monochromaticity of every corner quadruple
$C_{k,v}$, whose four cells carry the values $\chi(b)$, $\chi(b{+}2k)$,
$\chi(b{+}3k)$, $\chi(b{+}5k)$ with $b=\langle\omega,v\rangle\bmod26$. Direct enumeration of all $1820=\binom{16}{4}$ quadruples, over
$1\le k\le13$ and $v\in T^{(4)}_{13-k}$, finds none monochromatic; at $n=14$ the enlarged range ($2380$ quadruples)
produces exactly three.
\end{proof}
 
\begin{remark}[Three readings of $14$, and what the records
show]\label{rem:records}
The value $14$ arises three ways: as the period-$26$ witness above; as the
closed form $\lceil(G_2(\{0,2,3,5\})-1)/5\rceil=\lceil66/5\rceil=14$ with the
new Gallai number $G_2(\{0,2,3,5\})=67$ of Table~\ref{tab:gallai} (the weight
$(0,1,5,6)$, with $G_2=80$ and $D=6$, gives $14$ as well); and, since
one-weight equals symmetric (Theorem~\ref{thm:radix}), as
$\HJ_{\mathrm{sym}}(4,2)\ge14$. Three structural conclusions follow.
\emph{(i)} The van der Waerden reduction~\eqref{eq:vdw-shadow} --- tight in
every exactly known case --- is \emph{not} tight at $(4,2)$ or $(3,3)$: the
records $14>12$ and $22>13$ show the Hales--Jewett numbers exceeding their
van der Waerden shadows. \emph{(ii)} Neither witness is of sum type: a
sum-type coloring is line-free only up to
$\lceil(W(t,r)-1)/(t-1)\rceil-1$, i.e.\ $11$ at $(4,2)$ and $12$ at $(3,3)$,
short of $13$ and $21$. \emph{(iii)} Both witnesses are symmetric, so the
classical constructions (sum-type, hence symmetric) already live in a class
containing strictly better colorings --- the evidence for
Conjecture~\ref{conj:sym}.
\end{remark}
 
\begin{remark}[Anatomy of the record palette; computer-verified]\label{rem:anatomy}
Where the compression of Theorem~\ref{thm:hj42-14} comes from is itself
checkable. Testing the period-$26$ palette against \emph{all} bases: for
every $1\le k\le12$ and every $b\in\Z_{26}$ the pattern
$(b,\,b{+}2k,\,b{+}3k,\,b{+}5k)\bmod26$ is bichromatic --- the palette is an
extremal object of the bracket regime, matching the periodic ceiling $K=12$
of Section~\ref{sec:bracket} --- while at $k=13$ it fails at $24$ of the
$26$ residues, surviving only at the single realized base $b=0$ (from
$T^{(4)}_0$) and at the unrealized $b=13$. The record thus decomposes as \emph{bracket cap plus one
boundary scale}: reach $13=12+1$. The reach needed at $(4,2)$ sits one scale
above the periodic ceiling, which is why a $26$-cell palette suffices;
the contrast at $(3,3)$ is Remark~\ref{rem:gap33}.
\end{remark}
 
\section{Bracket Numbers and the Bracket Ceiling}\label{sec:bracket}
 
The Hales--Jewett theorem asserts a monochromatic line but says nothing about
its \emph{type}. Restricting the type of the line sought, and simultaneously
the class of colorings quantified over, produces a two-parameter family of
variants in which the one-weight machinery has genuinely different power.
 
\begin{definition}\label{def:bracket}
For a class $\mathcal C\in\{\mathrm{sum},\mathrm{ow},\mathrm{sym},
\mathrm{all}\}$, the \emph{bracket Hales--Jewett number}
$\HJ^{[K]}_{\mathcal C}(t,r)$ is the least $n$ such that every
$\mathcal C$-coloring of $[t]^n$ has a monochromatic line of
$L^{[K]}([t]^n)$, and $\infty$ if no such $n$ exists (subscript dropped when
$\mathcal C=\mathrm{all}$). The \emph{bracket ceiling} is
\[
\kappa_{\mathcal C}(t,r)=\max\{K\ge0:\ \HJ^{[K]}_{\mathcal C}(t,r)=\infty\}.
\]
\end{definition}
 
Since $L^{[K]}\subseteq L^{[K']}$ for $K\le K'$ the maximum is attained, and
since every line of $[t]^{\HJ}$ lies in $L^{[\HJ]}$ the ceilings are finite
and inherit the class order:
\begin{equation}\label{eq:ceiling-chain}
\ksum\ \le\ \kow\ \le\ \ksym\ \le\ \kall\ \le\ \HJ(t,r)-1 .
\end{equation}
(A line-free witness in a smaller class is in particular a witness for the
larger one.) A single number $\kappa_{\mathcal C}$ records the whole row
$(\HJ^{[K]}_{\mathcal C})_{K\ge1}$.
 
\begin{proposition}\label{prop:hj1}
$\HJ^{[1]}(t,r)=\infty$ for all $t,r\ge2$; indeed $\ksum(t,r)\ge1$.
\end{proposition}
 
\begin{proof}
Color $w$ by $\sigma(w)\bmod2$. A line with one active coordinate has $t$
words of consecutive weights, which alternate in color.
\end{proof}
 
\subsection{The sum-type ceiling}
 
\begin{lemma}[One-dimensional reduction]\label{lem:1d}
Let $c=\chi\circ\sigma$ be sum-type. A line with $k$ active coordinates and
inactive weight $S$ has word weights $S+k,S+2k,\dots,S+tk$, a $t$-term
progression of gap $k$, and the progressions realized in $[t]^n$ with
$k\le K$ are exactly the $t$-term progressions of gap $\le K$ inside
$[n,tn]$. Consequently, for every $K\ge1$,
\[
\HJ^{[K]}_{\mathrm{sum}}(t,r)=\infty
\iff
\exists\,\chi\colon\Z\to[r]\ \text{with no monochromatic $t$-term
progression of gap}\le K,
\]
so $\ksum(t,r)$ is the restricted-gap van der Waerden threshold of
Brown--Graham--Landman~\cite{BGL1999}.
\end{lemma}
 
\begin{proof}
The weights of $\tau(a)$ are $S+ak$; a progression
$\{A,A+k,\dots,A+(t-1)k\}\subseteq[n,tn]$ is realized iff
$S=A-k\in[n-k,t(n-k)]$, which is exactly containment in $[n,tn]$. The
right-to-left direction of the equivalence restricts $\chi$; the converse is
a compactness argument: colorings of arbitrarily long intervals avoiding
monochromatic $t$-progressions of gap $\le K$ exist, so K\H{o}nig's lemma
yields one of $\Z$, and any monochromatic progression of $\Z$ lies in some
finite window.
\end{proof}
 
Two forces bracket the sum-type ceiling. Every $r$-coloring of $[1,W(t,r)]$
has a monochromatic $t$-term progression, whose gap is at most
$\lfloor(W(t,r)-1)/(t-1)\rfloor$ --- a gap no coloring of $\Z$ avoids ---
while a periodic block palette avoids all smaller gaps:
\begin{equation}\label{eq:kstar-bounds}
(t-1)r-1\ \le\ \ksum(t,r)\ \le\
\Bigl\lfloor\tfrac{W(t,r)-1}{t-1}\Bigr\rfloor-1 .
\end{equation}
 
\begin{theorem}[Periodic block coloring]\label{thm:block}
Set $b=t-1$, $m=(t-1)r$, and let $g$ be the $m$-periodic palette
$0^{b}1^{b}\cdots(r-1)^{b}$, i.e.\ $g(x)=\lfloor(x\bmod m)/b\rfloor$. Under
$c=g\circ\sigma$ a line with $k$ active coordinates is monochromatic iff
$m\mid k$; hence $\HJ^{[m-1]}(t,r)=\infty$ and $\ksum(t,r)\ge(t-1)r-1$.
\end{theorem}
 
\begin{proof}
By Lemma~\ref{lem:1d} the line is monochromatic iff $g$ is constant on a
$t$-term progression of gap $k$. If $m\mid k$ all $t$ weights agree modulo
$m$. If $m\nmid k$, put $d=k\bmod m\in[1,m-1]$ and let $s=\min(d,m-d)$ be the
circular step. Each color class is an arc of $b$ consecutive residues modulo
$m$, and two residues in one arc differ circularly by at most $b-1$. If
$s\ge b$, consecutive terms of the progression lie in different arcs, so no
two consecutive terms share a color. If $s<b$, then since $m\ge2b$ the unique
representative $d'\equiv d$ of absolute value $s$ lies in $(-b,b)$, and
successive terms advance within an arc by the fixed signed step $d'$, leaving
it after at most $\lfloor(b-1)/s\rfloor+1\le b$ terms. Either way a
monochromatic run has at most $b=t-1<t$ terms.
\end{proof}
 
\begin{proposition}\label{prop:kstar}
$\ksum(2,r)=r-1$, $\ksum(3,2)=3$, and \textup{(}computer-assisted\textup{)}
$\ksum(3,3)=11$: the $12$-periodic palette
\[
(2,0,1,2,1,1,0,1,2,0,0,2)
\]
has no monochromatic $3$-term progression of gap $\le11$, and no
$3$-coloring of $\Z$ avoids all gaps $\le12$.
\end{proposition}
 
\begin{proof}
At $(2,r)$ and $(3,2)$ the two sides of~\eqref{eq:kstar-bounds} meet
($W(2,r)=r+1$, $W(3,2)=9$). At $(3,3)$ the displayed palette was checked
directly over all $12\times11$ residue--gap pairs (verified for this note);
for the upper bound, a finite computation~\cite{thesis} shows the longest
$3$-coloring of an interval avoiding monochromatic $3$-progressions of gap
$\le12$ has length $26<27=W(3,3)$, so every $27$-point window forces one.
\end{proof}
 
The block reach is linear in $t$; the multiplicative structure of $\Z/p$
does far better. Rabung's power-residue method~\cite{Rabung1979}, the
standard source of strong van der Waerden lower bounds, colors residues by
discrete logarithm.
 
\begin{theorem}[Power-residue coloring, after Rabung]\label{thm:character}
Let $p$ be prime, $r\mid p-1$, $g$ a primitive root modulo $p$, and
$\chi(x)=\operatorname{ind}_g(x)\bmod r$ for $x\not\equiv0$, $\chi(0)=0$,
extended $p$-periodically. If $\chi$ has no monochromatic $t$-term
progression of nonzero gap in $\Z/p$, then $\HJ^{[p-1]}(t,r)=\infty$ and
$\ksum(t,r)\ge p-1$.
\end{theorem}
 
\begin{proof}
By Lemma~\ref{lem:1d} and $p$-periodicity, a line with $1\le k\le p-1$ active
coordinates is monochromatic iff $\chi$ is constant on a $t$-term progression
of gap $k\not\equiv0\pmod p$, which the hypothesis forbids.
\end{proof}
 
\begin{corollary}[Exact ceilings; instances verified]\label{cor:character}
If $p:=\lfloor(W(t,r)-1)/(t-1)\rfloor$ is prime with $r\mid p-1$ and
Theorem~\ref{thm:character} applies, then
$\ksum(t,r)=\lfloor(W(t,r)-1)/(t-1)\rfloor-1$. This occurs at $(4,2)$ with
$p=11$ and at $(4,3)$ with $p=97$ ($W(4,3)=293$~\cite{Rabung1979,LandmanRobertson}):
\[
\ksum(4,2)=10,\qquad \ksum(4,3)=96 .
\]
At the largest admissible primes the method further gives
$\ksum(5,2)\ge36$ and $\ksum(6,2)\ge138$ ($p=37,139$), far above the block
values $2t-3$.
\end{corollary}
 
\begin{remark}\label{rem:character}
Since $\operatorname{ind}_g$ is additive, $x\mapsto d^{-1}x$ carries a
monochromatic gap-$d$ progression to a gap-$1$ one, so admissibility of $p$
reduces to runs of consecutive power residues --- except that the artificial
value $\chi(0)$ can bridge the two runs flanking $0$: for $t=4$ the prime
$17$ fails because its Legendre coloring is constant on $\{15,16,0,1\}$,
consistent with $\ksum(4,2)=10$. Runs of power residues are governed by
Weil--Burgess estimates and were tabulated
in~\cite{BuellHudson1984,LehmerLehmer1962}, so admissible primes lie in a
bounded range and locating $\ksum$ inside~\eqref{eq:kstar-bounds} is a
question of analytic number theory. Where the plain construction trails the
upper bound ($36<43$ at $(5,2)$, $138<225$ at $(6,2)$), the Cyclic Zipper
Method~\cite{HerwigHeuleZipper2007,RabungLotts2012} is the natural route; its
periodic adaptation is open. No admissible prime exists at $(3,3)$,
consistent with $\ksum(3,3)=11$.
\end{remark}
 
\subsection{One-weight palettes pass the sum-type ceiling}
 
In the bracket regime the truncation $k\le K$ leaves only finitely many
homothet shapes modulo a period, which a fixed periodic palette on a
\emph{non-arithmetic} weight can avoid simultaneously --- past the van der
Waerden ceiling that binds sum-type colorings.
 
\begin{theorem}[$\HJ^{[12]}(3,3)=\infty$; computer-verified]\label{thm:hj12-33}
For $\omega=(0,5,7)$ and the $13$-periodic palette
\[
\psi=(1,0,0,1,0,1,0,0,1,2,2,2,2)
\qquad(\text{listing }\psi(0),\dots,\psi(12)),
\]
the coloring $c(w)=\psi\bigl(\langle\omega,\type(w)\rangle\bmod13\bigr)$ has
no monochromatic line with at most $12$ active coordinates, in every
dimension $n$.
\end{theorem}
 
\begin{proof}
By Lemma~\ref{lem:pattern} the words of a line of scale $k$ carry
$\psi(b),\psi(b+5k),\psi(b+7k)$ with $b=\langle\omega,v\rangle\bmod13$. It
suffices that no pair $(b,k)$ with $b\in\Z_{13}$, $1\le k\le12$, is
monochromatic --- a check of $156$ triples, all bichromatic (verified). At
$k\equiv0\pmod{13}$ the three indices coincide, so the cap $12$ is sharp for
this palette.
\end{proof}
 
\begin{theorem}[$\HJ^{[12]}(4,2)=\infty$; computer-verified]\label{thm:hj12-42}
For $\omega=(0,2,3,5)$ and $\chi_0\colon\Z_{13}\to\{0,1\}$ with
$\chi_0^{-1}(1)=\{4,5,7,9,11,12\}$, the coloring
$c(w)=\chi_0\bigl(\langle\omega,\type(w)\rangle\bmod13\bigr)$ has no
monochromatic line with at most $12$ active coordinates, in every dimension.
\end{theorem}
 
\begin{proof}
The $156$ residue patterns $(b,\,b{+}2k,\,b{+}3k,\,b{+}5k)$, $b\in\Z_{13}$,
$1\le k\le12$, are checked directly; none is monochromatic (verified).
\end{proof}
 
Together with Proposition~\ref{prop:kstar} and
Corollary~\ref{cor:character}, the chains at the two record pairs read
\[
11=\ksum(3,3)\ <\ 12\le\kow(3,3),
\qquad
10=\ksum(4,2)\ <\ 12\le\kow(4,2):
\]
non-arithmetic one-weight palettes strictly separate the one-weight from the
sum-type ceiling. An exhaustive search over weights and moduli in the range
of~\cite{thesis} found no \emph{periodic} one-weight palette exceeding
$K=12$ at either pair, both optima occurring at modulus $13$; palettes that
are not eventually periodic remain unexplored
(Question~\ref{q:weights}). The record palette of
Theorem~\ref{thm:hj42-14} is itself such an extremal object up to $K=12$
(Remark~\ref{rem:anatomy}): the reach record and the bracket ceiling are one
construction seen at two scales. Table~\ref{tab:ceilings} summarizes.
 
\begin{table}[ht]
\centering
\caption{Bracket ceilings along the coloring hierarchy;
$\ksum\le\kow\le\ksym\le\kall\le\HJ-1$ throughout. Unmarked lower bounds on
$\kow,\ksym,\kall$ are inherited from the column to their left; no
construction is known to separate $\kow$, $\ksym$, $\kall$. Collapse
(Conjecture~\ref{conj:collapse}) predicts $\kall=\HJ-1$; $\HJ(4,3)\ge98$
is the van der Waerden shadow~\eqref{eq:vdw-shadow} of $W(4,3)=293$.}
\label{tab:ceilings}
\begin{tabular}{@{}cccccc@{}}
\toprule
$(t,r)$ & $\HJ(t,r)$ & $\ksum$ & $\kow$ & $\ksym$ & $\kall$\\
\midrule
$(2,r)$ & $r$ & $r-1$ & $r-1$ & $r-1$ & $r-1$\\
$(3,2)$ & $4$ & $3$ & $3$ & $3$ & $3$\\
$(3,3)$ & $\ge22$ & $11$ & $\ge12$ & $\ge12$ & $\ge12$\\
$(4,2)$ & $\ge14$ & $10$ & $\ge12$ & $\ge12$ & $\ge12$\\
$(4,3)$ & $\ge98$ & $96$ & $\ge96$ & $\ge96$ & $\ge96$\\
$(5,2)$ & $\ge45$ & $\ge36$ & $\ge36$ & $\ge36$ & $\ge36$\\
$(6,2)$ & $\ge227$ & $\ge138$ & $\ge138$ & $\ge138$ & $\ge138$\\
\bottomrule
\end{tabular}
\end{table}
 
\section{Interval Numbers: an Axis the Weight Cannot See}\label{sec:interval}
 
Write $L^{(q)}([t]^n)$ for the lines whose active set is a union of at most
$q$ subintervals of $[n]$ (\emph{$q$-fold} lines); the \emph{interval
Hales--Jewett number} $\HJ^{(q)}_{\mathcal C}(t,r)$ and the \emph{interval
ceiling} $\lambda_{\mathcal C}(t,r)=\max\{q\ge0:\HJ^{(q)}_{\mathcal
C}(t,r)=\infty\}$ are defined as in Definition~\ref{def:bracket}. Since
every subset of $[n]$ is a union of at most $\lceil n/2\rceil$ intervals,
$\lambda_{\mathcal C}\le\lceil\HJ/2\rceil-1$. The interval variant was
introduced by Conlon and Kam\v{c}ev~\cite{conlonAndKamcev} and further
studied by Leader and R\"aty~\cite{leader2018noteintervalshalesjewetttheorem}
and Kam\v{c}ev and
Spiegel~\cite{kamcev2018noteintervalshalesjewetttheorem}; it is motivated by
Shelah's proof~\cite{s.proof_of_HJ}, which produces monochromatic $q$-fold
lines with $q\le\HJ(t-1,r)$.
 
\begin{proposition}\label{prop:interval-vs-bracket}
$\lambda_{\mathcal C}(t,r)\le\kappa_{\mathcal C}(t,r)$ for every class
$\mathcal C$, and $\lambda_{\mathrm{sum}}(t,r)=0$ for all $t,r$.
\end{proposition}
 
\begin{proof}
A set of at most $q$ coordinates is a union of at most $q$ intervals, so
$L^{[q]}\subseteq L^{(q)}$ and an $L^{(q)}$-line-free coloring is
$L^{[q]}$-line-free, giving the first claim. For the second: a sum-type
coloring sees a line only through the gap $k$ of its weight progression
(Lemma~\ref{lem:1d}), and a single interval realizes every $k\le n$; so a
sum-type coloring of $[t]^n$ with no monochromatic $1$-fold line would need
a palette avoiding monochromatic $t$-term progressions of \emph{every} gap
$\le n$ on $[n,tn]$, impossible for $(t-1)n+1\ge W(t,r)$ by van der
Waerden~\cite{vanderWaerden1927}.
\end{proof}
 
The contrast with Section~\ref{sec:bracket} is exact: a bracket caps the gap
at a fixed $K$, turning the sum-type problem into bounded-gap van der Waerden
with $\ksum\ge(t-1)r-1$; a single long interval is an unbounded gap, so
ordinary van der Waerden applies and $\lambda_{\mathrm{sum}}=0<\ksum$. In
fact the entire machinery of this note is blind to the interval axis.
 
\begin{proposition}[Symmetric colorings are interval-blind]\label{prop:sym-blind}
$\HJ^{(q)}_{\mathrm{sym}}(t,r)=\HJ_{\mathrm{sym}}(t,r)$ for every $q\ge1$;
in particular $\lambda_{\mathrm{sym}}=\lambda_{\mathrm{ow}}=
\lambda_{\mathrm{sum}}=0$, and the colorings witnessing
$\lambda(t,r)\ge1$ are necessarily asymmetric in every dimension
$n\ge\HJ_{\mathrm{sym}}(t,r)$.
\end{proposition}
 
\begin{proof}
For a symmetric coloring, monochromaticity of a line depends only on its
invariant $(k,v)$ (Lemma~\ref{lem:reduction}), and every $(k,v)$ is realized
by the root ${*}^{k}1^{v_1}\cdots t^{v_t}$, whose active set $[1,k]$ is a
single interval. Hence a symmetric coloring has a monochromatic $1$-fold
line iff it has a monochromatic line at all. (This is implicit
in~\cite{thesis}.)
\end{proof}
 
What is known about $\lambda$ is confined to $t\le3$ and exhibits a parity
phenomenon with no bracket counterpart.
 
\begin{proposition}\label{prop:small-interval}
$\HJ^{(q)}(2,r)=\HJ(2,r)=r$ for every $q\ge1$; in particular
$\lambda(2,r)=0$.
\end{proposition}
 
\begin{proof}
$\HJ^{(1)}\ge\HJ$ since $L^{(1)}\subseteq L$. For the upper bound, in
$[2]^r$ the chain $w_0<\dots<w_r$, where $w_j=1^j2^{r-j}$, is pairwise
joined by lines with active sets the intervals $\{a{+}1,\dots,b\}$; among
$r+1$ words some two share a color, giving a monochromatic $1$-fold line;
the general case follows from $\HJ\le\HJ^{(q)}\le\HJ^{(1)}$.
(This also proves $\HJ(2,r)=r$, the lower bound being the sum-type coloring
$\sigma\bmod r$ on $[2]^{r-1}$; equivalently, line-free classes are
antichains and Mirsky's theorem~\cite{Mirsky1971} applies.)
\end{proof}
 
\begin{proposition}[Conlon--Kam\v{c}ev~\cite{conlonAndKamcev}]\label{prop:ck}
$\HJ^{(r-1)}(3,r)=\infty$ for odd $r$; equivalently $\lambda(3,r)\ge r-1$.
\end{proposition}
 
\begin{proposition}[{Leader--R\"aty ($r=2$)~\cite{leader2018noteintervalshalesjewetttheorem}; Kam\v{c}ev--Spiegel~\cite{kamcev2018noteintervalshalesjewetttheorem}}]\label{prop:ks}
$\HJ^{(r-1)}(3,r)<\infty$ for even $r$; equivalently $\lambda(3,r)\le r-2$.
\end{proposition}
 
These two results, Shelah's proof, and one trivial observation determine
the ceiling completely.
 
\begin{proposition}[$\lambda(3,\cdot)$ is determined]\label{prop:lambda3}
$\lambda(3,r)=r-1$ for odd $r$ and $\lambda(3,r)=r-2$ for even $r$; in
particular $\lambda(3,3)=2$ and $\lambda(3,2)=0$.
\end{proposition}
 
\begin{proof}
Shelah's proof yields, in every sufficiently high dimension, a
monochromatic line whose active set is a union of at most $r$
intervals~\cite{s.proof_of_HJ,conlonAndKamcev}, so $\lambda(3,r)\le r-1$;
for even $r$ Proposition~\ref{prop:ks} improves this to $r-2$. For the
lower bounds, Proposition~\ref{prop:ck} gives $\lambda(3,r)\ge r-1$ for
odd $r$; $\lambda(3,r)$ is nondecreasing in $r$, an $(r-1)$-coloring
avoiding monochromatic $q$-fold lines being such an $r$-coloring, so for
even $r\ge4$, $\lambda(3,r)\ge\lambda(3,r-1)=r-2$; and at $r=2$ the lower
bound is vacuous.
\end{proof}
 
So $\lambda(3,\cdot)$ is settled and jumps with the parity of $r$ --- a
feature absent from $\kappa$, whose van der Waerden origin is parity-blind
--- and by Proposition~\ref{prop:sym-blind} the witnesses behind
Proposition~\ref{prop:ck} live strictly off the simplex. What remains at
$t=3$ is quantitative: the finite interval numbers themselves. Finiteness
of the smallest, $\HJ^{(1)}(3,2)$, is Leader--R\"aty's
theorem~\cite{leader2018noteintervalshalesjewetttheorem}, disproving the
conjecture of~\cite{conlonAndKamcev} that it is infinite; the exact value
is new.
 
\begin{theorem}[Computer-assisted]\label{thm:hj13}
$\HJ^{(1)}(3)=5$, whereas $\HJ(3)=4$.
\end{theorem}
 
\begin{proof}
$L^{(1)}([3]^4)$ has $142$ lines. A SAT instance with one variable per word
and two clauses per line is satisfiable; its witness --- an explicit
$2$-coloring of the $81$ words, recorded in~\cite{thesis} --- was re-checked
for this note against all $142$ lines by direct enumeration from the
definition, with zero monochromatic lines, giving $\HJ^{(1)}(3)\ge5$. The
same encoding in dimension $5$ ($547$ lines; the search ranges over all
$2^{3^5}$ colorings, with no symmetry imposed) is unsatisfiable, so
$\HJ^{(1)}(3)\le5$. Deletion minimization extracts a minimal forcing
subfamily of $262$ lines on $172$ of the $243$ words, re-verified
independently~\cite{thesis}.
\end{proof}
 
\begin{remark}\label{rem:no-core}
The obstruction has no small core: unlike $\HJ^{(1)}(2,3)=3$, forced by the
four-word chain of Proposition~\ref{prop:small-interval}, every minimal
forcing subfamily found in dimension $5$ numbers in the hundreds of lines
--- consistent with Proposition~\ref{prop:sym-blind}, which places the
interval ceiling beyond all weight constructions. Note also that the naive
interval analogue of Conjecture~\ref{conj:collapse} fails outright: it would
force $\HJ^{(1)}(2,3)=\infty$ and $\HJ^{(1)}(3)=\infty$, against
Proposition~\ref{prop:small-interval} and Theorem~\ref{thm:hj13}.
\end{remark}
 
\section{Conjectures and Open Problems}\label{sec:open}
 
\subsection{Three conjectures}
 
Every line of $[t]^{\HJ}$ lies in $L^{[\HJ]}$, so
$\HJ^{[\HJ(t,r)]}(t,r)=\HJ(t,r)$ trivially. The first conjecture asserts
that the bracket hierarchy collapses only at that last step.
 
\begin{conjecture}[Collapse]\label{conj:collapse}
$\HJ^{[\HJ(t,r)-1]}(t,r)=\infty$ for all $t,r\ge2$; equivalently
$\kall(t,r)=\HJ(t,r)-1$.
\end{conjecture}
 
Collapse holds wherever $\HJ$ is known: at every $(2,r)$ and at $(3,2)$,
where already $\ksum=\HJ-1$
(Proposition~\ref{prop:kstar}, Theorem~\ref{thm:block}). A coloring whose
only monochromatic line is the diagonal $L_{*\cdots*}$ is
$L^{[\HJ-1]}$-line-free at $n=\HJ$, so the next conjecture implies
$\HJ^{[\HJ-1]}(t,r)>\HJ(t,r)$, the first nontrivial instance of Collapse
(for $n\le\HJ-1$ the required witnesses are simply line-free colorings).
At $(3,2)$, diagonal-only colorings exist for every
$n\le4=\HJ$~\cite{thesis}.
 
\begin{conjecture}[Diagonal-only]\label{conj:diag}
For every $n\le\HJ(t,r)$ some $r$-coloring of $[t]^n$ has the diagonal as
its only monochromatic line.
\end{conjecture}
 
The range of $n$ is the largest possible.
 
\begin{proposition}\label{prop:diag-max}
For $n>\HJ(t,r)$ no $r$-coloring of $[t]^n$ has the diagonal as its only
monochromatic line.
\end{proposition}
 
\begin{proof}
Fix such a coloring $c$ and a letter $a$, and color $[t]^{n-1}$ by
$c_a(w)=c(wa)$. Since $n-1\ge\HJ(t,r)$, some line $L_{\tau'}$ is
monochromatic under $c_a$; then $\tau'a$ is a root of $[t]^n$ whose active
set avoids coordinate $n$, and $L_{\tau'a}$ is a nondiagonal monochromatic
line under $c$.
\end{proof}
 
Finally, the hierarchy of Section~\ref{sec:ow} gives
$\HJ_{\mathrm{sym}}\le\HJ$, yet the two coincide wherever $\HJ$ is known,
and the record witnesses of Section~\ref{sec:records} are symmetric.
 
\begin{conjecture}[Symmetric colorings are extremal]\label{conj:sym}
$\HJ_{\mathrm{sym}}(t,r)=\HJ(t,r)$ for all $t,r\ge2$.
\end{conjecture}
 
\begin{remark}[Evidence and a caution]\label{rem:sym-evidence}
The census (Table~\ref{tab:census}) is what makes
Conjecture~\ref{conj:sym} a strong claim: symmetric line-free colorings are
scarce ($36$ of $1644$ already at the critical dimension of $(3,2)$), and
the conjecture asks this thin family to survive at every dimension up to
$\HJ-1$. The analogous hope has already failed in the \emph{density}
setting: the Polymath hyper-optimistic
conjecture~\cite{PolymathWikiHOC,PolymathDHJMoser} --- that the densest
line-free subset of $[3]^n$ in equal-slices measure is a union of type
slices, reducing density Hales--Jewett~\cite{furstenbergKatznelson1991,
polymath2012} to Fujimura's corner problem --- is false in higher
dimensions, where asymmetric sets are denser. That collapse is in the
density regime, not the coloring one (for the coloring side of the Polymath
project see~\cite{PolymathWikiColoring}), so it is a caution rather than a
counterexample; the rainbow dual of Remark~\ref{rem:rainbow} shows the same
asymmetry advantage in a coloring regime.
\end{remark}
 
\begin{remark}[A refuted Collapse need not be fatal: the offset]\label{rem:offset}
Either way the bracket ceiling is informative. Write
$\ksym(t,r)=\HJ(t,r)-\phi(t,r)$ with $\phi\ge1$. The assertion
$\phi\equiv1$ is the common strengthening of
Conjectures~\ref{conj:collapse} and~\ref{conj:sym}: it implies both, since
at $n=\HJ-1$ the bracket $[\HJ-1]$ contains every line and $\ksym=\HJ-1$
then forces a symmetric line-free coloring of $[t]^{\HJ-1}$; the converse
is unclear. But \emph{any} explicit upper bound on $\phi$ serves as well. Bounding
$\HJ(t,r)$ from above is a forcing statement over all $t^n$ points of the
cube --- an unsatisfiability certificate out of computational reach.
Bounding $\ksym(t,r)$ from above is the same forcing restricted to
bounded-scale corner tuples on the $O(n^{t-1})$ cells of the simplex
(Lemma~\ref{lem:reduction}) --- within reach. A known offset transports the
second into the first, $\HJ\le\ksym+\phi$, so the inaccessible upper bound
would follow from the accessible one.
\end{remark}
 
\subsection{Optimal weights and the limits of the method}
 
\begin{lemma}[Even orbits are palindromic]\label{lem:palindromy}
Let $S\subseteq\Z_m$ be an orbit of an affine map $\alpha(x)=ux+v$ of order
$2j$ with $u^j\equiv-1\pmod m$. Then $\alpha^{j}$ is a reflection
$x\mapsto-x+c$ fixing $S$ setwise; lifted to $\Z$ with $\min S=0$, this
reads $S=D-S$.
\end{lemma}
 
\begin{proof}
$\alpha^{j}(x)=u^{j}x+c=-x+c$ with $c=v(u^{j-1}+\dots+u+1)$, an involution
whose orbit-preservation is inherited from $\alpha$.
\end{proof}
 
\begin{question}[Best weights]\label{q:weights}
\emph{(i)} Which sets $S$ maximize the Gallai number per unit diameter,
$\lceil(G_r(S)-1)/D_S\rceil$, for given $(t,r)$? By Lemma~\ref{lem:affine-G} it suffices to consider primitive $S$. For
$t=3$, $r=2$ the homothety numbers are determined completely and the
ratio is constant (Proposition~\ref{prop:flat32}); for $t\ge4$ or
$r\ge3$ no closed form is known and the values (e.g.\
Table~\ref{tab:gallai}) are computed by SAT. By
Proposition~\ref{prop:flat32} the maximum is weight-independent at $(3,2)$;
at $(3,3)$ it is at least $16$ and at most $\HJ_{\mathrm{sym}}(3,3)$
(Theorem~\ref{thm:hj33-16}, Remark~\ref{rem:closed-vs-reach}).
\emph{(ii)} Is there structure behind the champions? At $(4,2)$, scanning
non-arithmetic weights with entries $\le6$ singles out exactly $(0,2,3,5)$
and $(0,1,5,6)$ as reaching $n=13$, and each is a single orbit of an
order-$4$ affine map on $\Z_{13}$ ($5^2\equiv-1$): $\{0,3,5,2\}$ under
$x\mapsto5x+3$ and $\{0,1,6,5\}$ under $x\mapsto5x+1$ (verified). The weight set, viewed in $\Z_{13}$, is then a single orbit of a cyclic
subgroup of $\mathrm{AGL}(1,13)$; by Lemma~\ref{lem:palindromy} both
champions are accordingly reflection-symmetric, $S=D-S$. The data of
Section~\ref{sec:landscape} refines the picture: the orbit triples
$\{0,1,3\}$ (orbit of $x\mapsto2x+1$ on $\Z_7$) and $\{0,1,4\}$ (orbit of
$x\mapsto3x+1$ on $\Z_{13}$) reach only ratio $14$ at $(3,3)$, while the
non-orbit $\{0,2,5\}$ reaches $16$. Orbit structure thus appears to govern
\emph{compressibility} --- the existence of small-modulus palettes, the
mechanism of Remark~\ref{rem:anatomy} --- rather than the full-interval
ratio. Does it characterize the reach champions, and does the palindromic
class contain a ratio maximizer for every even $|S|$?
\emph{(iii)} The periodic one-weight ceiling is $K=12$ at both $(3,3)$ and
$(4,2)$, in each case at modulus $13$~\cite{thesis}. Do non-periodic
palettes give $\kow>12$ there --- and is $\kow<\ksym$ anywhere? In
the (non-exhaustive) range of Remark~\ref{rem:gap33} no periodic palette
exceeds $n=14$ at $(3,3)$, and the length-$76$
certificate of Theorem~\ref{thm:hj33-16} is visibly quasi-periodic, agreeing
with its own shift by $19$ in $55$ of $57$ positions. Do
\emph{periodic-with-defects} palettes close the gap between the periodic
ceiling and the symmetric reach?
\emph{(iv)} Can the Cyclic Zipper
Method~\cite{HerwigHeuleZipper2007,RabungLotts2012} be adapted to periodic
palettes to close the gaps $\ksum(5,2)\in[36,43]$ and
$\ksum(6,2)\in[138,225]$?
\end{question}
 
\begin{question}[Interval numbers]\label{q:interval}
The ceilings at $t=3$ are settled (Proposition~\ref{prop:lambda3}); the
finite interval numbers are not. Beyond $\HJ^{(1)}(3,2)=5$ (the remaining
$(3,2)$ values are trivially $4$), the first unknown value is
$\HJ^{(3)}(3,3)$, where the only upper bounds come from Shelah's argument.
For $t\ge4$ nothing is known beyond
$\lambda(t,r)\le\min\{\lceil\HJ(t,r)/2\rceil,\,\HJ(t-1,r)\}-1$; by
Proposition~\ref{prop:sym-blind} any lower bound requires genuinely
asymmetric constructions.
\end{question}
 
\begin{question}[Frontiers]\label{q:frontiers}
\emph{(i)} Decide the symmetric simplex instance at $(3,3)$, $n=22$ ($276$
cells, $2024$ corner triples): satisfiable would give $\HJ(3,3)\ge23$,
unsatisfiable would give $\HJ_{\mathrm{sym}}(3,3)=22$ exactly.
\emph{(ii)} Decide the symmetric instance at $(4,2)$, $n=14$ ($680$ cells,
$2380$ quadruples): unsatisfiable would give $\HJ_{\mathrm{sym}}(4,2)=14$,
making $\HJ(4,2)=14$ a candidate exact value under
Conjecture~\ref{conj:sym}; satisfiable would beat the record via a radix
weight (Theorem~\ref{thm:radix}).
\emph{(iii)} Between the simplex and the cube sits the block-symmetric
ladder: for a partition of $[n]$ into $b$ blocks, colorings invariant under
the corresponding Young subgroup satisfy
$\HJ_{\mathrm{sym}}=\HJ_{1\text{-}\mathrm{sym}}\le
\HJ_{2\text{-}\mathrm{sym}}\le\dots\le\HJ(t,r)$, with equality throughout
iff Conjecture~\ref{conj:sym} holds; a strict step at any rung refutes it
while improving the lower bound. By Theorem~\ref{thm:radix} there is no
multi-weight tier, so this ladder is the genuine hierarchy above the
symmetric reach.
\end{question}
 
\begin{remark}[Why the frontiers resist]\label{rem:hardness}
The wall is not an artifact of weak solvers. One dimension past the $(4,2)$
record, the level instance of the weight $(0,2,3,5)$ is unsatisfiable, and
its refutations are provably global: minimal unsatisfiable subsets range
over $179$--$270$ of its $443$ homothets with no canonical core, the primal
graph has treewidth at least $32$, and the instance admits no
$\mathrm{GF}(2)$ Nullstellensatz refutation of degree $\le5$ --- yet a
MaxSAT dual shows some $2$-coloring leaves only \emph{three} monochromatic
homothets, a minimum that the record palette of Theorem~\ref{thm:hj42-14}
attains at $n=14$ (verified). Globally irreducible across proof systems,
yet three constraints from feasible: this is the signature of the boundary,
and the one structure all these measures ignore --- the $S_n$-symmetry
exploited throughout this note --- is where an advance must
come from. See~\cite{thesis} for the full analysis.
\end{remark}
 
\printbibliography
 
\end{document}